\documentclass[12pt,reqno]{amsart}
\usepackage{latexsym,amsmath,amssymb, amsthm, mathscinet}
\usepackage{cases, verbatim}

\usepackage{amsfonts,amsmath,amsthm, amssymb}
\usepackage{cases}
\usepackage[usenames]{color}
\usepackage{enumerate}

\newtheorem{result}{\textbf{Theorem}}

\newtheorem{lemma}{Lemma}[section]
\newtheorem{definition}{Definition}[section]

\newtheorem{remark}{Remark}[section]
\newtheorem{example}{Example}[section]
\numberwithin{equation}{section}

\newcommand{\N}{\mathbb{N}}

\newcommand{\R}{\mathbb{R}}

\newcommand{\T}{\mathbb{T}}

\newcommand{\cC}{\mathcal{C}}

\newcommand{\AC}{{\rm AC\,}}

\newcommand{\BUC}{{\rm BUC\,}}

\newcommand{\Li}{L^{\infty}}
\newcommand{\Lip}{{\rm Lip\,}}

\newcommand{\ep}{\varepsilon}

\begin{document}
\title[Rate of convergence for homogenization]{Rate of convergence for homogenization of nonlinear weakly coupled Hamilton-Jacobi systems}

\author{Hiroyoshi Mitake}
\address[H. Mitake]{
	Graduate School of Mathematical Sciences,
	University of Tokyo
	3-8-1 Komaba, Meguro-ku, Tokyo, 153-8914, Japan}
\email{mitake@g.ecc.u-tokyo.ac.jp}

\author{Panrui Ni}
\address[P. Ni]{
Department 1: Graduate School of Mathematical Sciences, University of Tokyo, 3-8-1 Komaba, Meguro-ku, Tokyo 153-8914, Japan; Department 2: Shanghai Center for Mathematical Sciences, Fudan University, Shanghai 200438, China}
\email{panruini@g.ecc.u-tokyo.ac.jp}


\makeatletter
\@namedef{subjclassname@2020}{\textup{2020} Mathematics Subject Classification}
\makeatother

\date{\today}

\begin{abstract}
Here, we study the periodic homogenization problem of
nonlinear weakly coupled systems of Hamilton-Jacobi equations
in the convex setting.
We establish a rate of convergence $O(\sqrt{\varepsilon})$ which is sharp.
\end{abstract}

\date{\today}

%
%
%

%

\maketitle


\section{Introduction}

Periodic homogenization for coercive Hamilton-Jacobi equations has been first proved in \cite{LPV}. In recent years there has been much interest on quantitative homogenization for first-order Hamilton--Jacobi equations.
For convex case, the lower bound $u^\varepsilon-\bar u\geqslant -C\varepsilon$ is obtained in \cite{MTY}, as well as the upper bound $u^\varepsilon-\bar u\leqslant C\varepsilon$ under restricted assumptions by using weak KAM theory.
More recently, the optimal rate of convergence $O(\varepsilon)$ in the convex setting has been obtained in  \cite{TY} by using the curve cutting lemma. The method in \cite{TY} is generalized to study Hamilton--Jacobi equations with multiscale in \cite{HJ}.
It is worth noting that for a general non-convex setting, the best known rate is still $O(\varepsilon^{1/3})$ obtained by \cite{CDI}.
We refer to  \cite{C,HJMT,HTZ,MS,NT,QSTY,T}, and the references therein for other development of this direction.

In this paper we are concerned with the weakly coupled systems of Hamilton--Jacobi equations with a nonlinear coupling of the form:
\begin{equation}\label{E}\tag{E$_\varepsilon$}
  \left\{
   \begin{aligned}
   &\partial_t u_i^\ep(x,t)+H_i\Big(x,\frac{x}{\varepsilon},Du_i^\ep(x,t),\mathbf{u}^\ep(x,t)\Big)=0\quad \text{for}\ (x,t)\in\mathbb R^n\times(0,T), \  i\in\{1,\dots,m\},
   \\
   &u_i^\ep(x,0)=\varphi_i(x) \quad\text{for} \ x\in\R^n, \
   i\in\{1,\dots,m\},
   \end{aligned}
   \right.
\end{equation}
where
$\ep>0$, $T>0$, $n, m\in\N$, and
$\mathbf{H}=(H_1,\ldots,H_m):\mathbb R^n\times \mathbb T^n\times\mathbb R^n\times\mathbb R^m\to\mathbb R^m$,
$\mathbf{u}^\ep=(u^\ep_1,\ldots,u^\ep_m):\R^n\times[0,T]\to\R^m$ is the unknown function and $\boldsymbol{\varphi}:=(\varphi_1,\ldots,\varphi_m):\R^n\to\R^m$ is a given continuous $\R^m$-valued function. Here, we denote by $\mathbb T^n$ the standard $n$-dimensional torus.
\textit{Throughout} the paper we assume 
that for all $i\in\{1,\ldots,m\}$, 
\begin{itemize}
\item[(H1)]
$H_i\in \textrm{BUC}(\mathbb R^n\times \mathbb T^n\times B(0,R)\times\mathbb R^m)$ for each $R>0$, $H_i(x,y,p,\mathbf{u})=H_i(x,y+k,p,\mathbf{u})$ for all $k\in\mathbb Z^n$, and 
$\varphi_i\in \text{Lip}(\mathbb R^n)\cap \textrm{BUC}(\mathbb R^n)$, \item[(H2)] $\lim_{|p|\to+\infty}\inf_{x\in\mathbb R^n,y\in\mathbb T^n}H_i(x,y,p,\mathbf{0})=+\infty$,
\item[(H3)] $p\mapsto H_i(x,y,p,\mathbf{u})$ is convex for each $(x,y,\mathbf{u})\in \mathbb R^n\times\mathbb T^n\times\mathbb R^m$,
\item[(H4)] there is Lip$(H)>0$ such that $|H_i(x_1,y,p,\mathbf{u})-H_i(x_2,y,p,\mathbf{u})|\leqslant \textrm{Lip}(H)|x_1-x_2|$ for any $x_1,x_2\in\mathbb R^n$ and $(y,p,\mathbf{u})\in\mathbb T^n\times\mathbb R^n\times\mathbb R^m$,
\item[(H5)]
there is $\Theta>0$ such that 
\[
|H_i(x,y,p,\mathbf{u}^1)-H_i(x,y,p,\mathbf{u}^2)|\leqslant 
\Theta \max_{j\in\{1,\ldots,m\}}|u^1_j-u^2_j|
\] 
for any $\mathbf{u}^1,\mathbf{u}^2\in \mathbb R^m$ and $(x,y,p)\in \mathbb R^n\times\mathbb T^n\times\mathbb R^n$.
\end{itemize}
In this paper, we denote by BUC$(X)$ the set of bounded and uniformly continuous functions on a metric space $X$, and Lip$(X)$ the set of Lipschitz continuous functions on $X$. 
We also denote by $C(X;\mathbb R^m)$ (resp., $\text{Lip}(X;\mathbb R^m)$) the set of all continuous (resp., Lipschitz continuous) $\mathbb R^m$-valued functions on $X$. 
For $\textbf{u}:X\to\mathbb R^m$, we define $|\textbf{u}(x)|:=\max_{i\in \{1,\dots, m\}}|u_i(x)|$ for $x\in X$ and $\|\textbf{u}\|_{L^\infty}:=\max_{i\in \{1,\dots, m\}}\sup_{x\in X}|u_i(x)|$. 
Note that combining (H2) with (H5), we see that $H_i(x,y,p,\mathbf{u})$ is uniformly coercive on any compact set of $\mathbf{u}\in\mathbb R^m$.
We are always concerned with viscosity solutions,
and the adjective ``viscosity" is often omitted in the paper. Note that there might be unbounded solutions of \eqref{E} and \eqref{E0}. In this paper, we are only concerned with the unique solution of \eqref{E} and \eqref{E0} in BUC$(\mathbb{R}^n\times [0,T];\mathbb{R}^m)$. Similarly, we are only concerned with the unique solution of \eqref{hje} and \eqref{hj0} in BUC$(\mathbb{R}^n;\mathbb{R}^m)$.

Weakly coupled systems arise as the dynamic programming for the optimal control, random Lax--Oleinik semigroups (see \cite{DSZ, MSTY} for instance), and the well-posedness of viscosity solutions is well-established in \cite{IK}.
The convergence of homogenization of monotone weakly coupled systems of Hamilton-Jacobi equations is proved in \cite{CLL}, 
see \cite{WZ} for the non-coercive case. 
See also \cite{CM} for the second order case with a monotone nonlinear coupling, 
and see \cite{Evans2,F,MT,S} 
for the linear coupling with a penalization 
of order $\varepsilon^{-1}$. 
Here, we establish a rate of convergence $O(\sqrt{\varepsilon})$ of weakly coupled systems \eqref{E} of Hamilton-Jacobi equations with a nonlinear coupling, which may not be monotone.
To the best of our knowledge, it is the first result about the rate of convergence for homogenization of nonlinear weakly coupled systems of Hamilton-Jacobi equations of the form \eqref{E}.
We also refer to \cite{MT} for a rate of convergence for homogenization of
Hamilton--Jacobi systems with linear coupling of the order $\ep^{-1}$.
System \eqref{E} with condition (H5) on the coupling is rather general. In addition, when \eqref{E} reduces to a single equation, the result obtained in this paper also gives the rate of convergence for homogenization of single Hamilton--Jacobi equations depending Lipschitz continuously on the unknown function. For related works on this kind of equations, one can refer to \cite{IWWY,WWY}.

Moreover, we consider a stationary system \eqref{hje} with a discounted term. This kind of systems have already been discussed in \cite{DZ,IJ}. We  provide a rate of convergence $O(\sqrt{\varepsilon})$ for homogenization of this stationary system.
There have been other works on the asymptotic problems for weakly coupled systems of Hamilton--Jacobi equations.
We refer to \cite{DZ, I, IJ} for the vanishing discount problems and to \cite{CGMT, CLLN} for the large-time behavior of solutions of the Cauchy problem.
For non-monotone coupling, one can refer to \cite{JWY,N}.



\medskip
Here, we define the effective Hamiltonian $\bar{\mathbf{H}}:\mathbb R^n \times\mathbb R^n\times\mathbb R^m\to\mathbb R^m$ corresponding to \eqref{E} following \cite[Section 4]{CLL}.
\begin{definition}
For each $\mathbf{c}\in\mathbb R^m$, $x\in \mathbb R^n$, $p\in\mathbb R^n$ and $i\in\{1\dots,m\}$, there exists a unique constant $\bar H_i(x,p,\mathbf{c})$ such that there exists a solution $v$ of
\begin{equation}\label{eq:cell}
H_i(x,y,p+Dv(y),\mathbf{c})=\bar H_i(x,p,\mathbf{c})\quad \textrm{for}\ y\in \mathbb T^n.
\end{equation}
Here, we note that the equations above are decoupled.
\end{definition}
We establish the rate of convergence of the solution $\mathbf{u}^\varepsilon$ of \eqref{E} to the solution $\bar{\mathbf{u}}$ of
\begin{equation}\label{E0}\tag{$\bar{\textrm{E}}$}
  \left\{
   \begin{aligned}
   &\partial_t \bar{u}_i(x,t)+\bar{H}_i(x,D\bar{u}_i(x,t),\bar{\mathbf{u}}(x,t))=0\quad \textrm{for}\ (x,t)\in\mathbb R^n\times(0,T), i\in\{1,\ldots,m\},
   \\
   &\bar{u}_i(x,0)=\varphi_i(x)
\quad \textrm{for} \ x\in\mathbb R^n, \  i\in\{1,\dots,m\}.
   \\
   \end{aligned}
   \right.
\end{equation}
Here $\bar{\mathbf{u}}=(\bar{u}_1,\ldots,\bar{u}_m):\R^n\times[0,T]\to\R^m$ is the unknown function, and functions $\varphi_1,\ldots,\varphi_m$ are given as before.
The first main result of the paper is as follows.
\begin{result}\label{thm1}
For each $\varepsilon>0$, there is a unique solution of both \eqref{E} and \eqref{E0}, which is denoted by $\mathbf{u}^\varepsilon$ and $\bar{\mathbf{u}}$, respectively. There is a constant $C>0$ depending only on $n$, $\mathbf{H}$, $\|\varphi_i\|_{W^{1,\infty}}$ and $T$ such that
for all $t\in[\sqrt{\varepsilon}, T]$ we have
\begin{equation*}
|\mathbf{u}^\varepsilon(x,t)-\bar{\mathbf{u}}(x,t)|\leqslant C\sqrt{\varepsilon}
\quad\textrm{for all} \ x\in\R^n,
\end{equation*}
and for all $t\in(0,\sqrt{\varepsilon})$ we have
\begin{equation}\label{r2}
  |\mathbf{u}^\varepsilon(x,t)-\bar{\mathbf{u}}(x,t)|\leqslant C\min\{t,\varepsilon\}
\quad\textrm{for all} \ x\in\R^n.
\end{equation}
\end{result}

Our approach to prove Theorem \ref{thm1} is to introduce an iteration argument to approximate the solution to \eqref{E} as described in Section \ref{sec:thm1}.  By using a recent development of quantitative homogenization for Hamilton--Jacobi equations in \cite{TY, HJ},
we establish a uniform estimate for approximations defined by \eqref{func:ul} of the solution to \eqref{E} in Lemma \ref{un}, which is a key ingredient in our paper.

It is worth emphasizing that because of the dependence of $\mathbf{H}$ on $\mathbf{u}^\ep$ in \eqref{E}, it is not clear to see if we can directly apply the method introduced in \cite{CDI} by using the discount approximation works to obtain a rate of convergence of homogenization problem \eqref{E}. This quantitative problem has been untouched for systems of the form \eqref{E}. We believe that our iteration argument is necessary to obtain a rate of convergence when we consider \eqref{E}.
For instance, one can find an analogous way to construct an auxiliary function associated with \eqref{u1eq} below, and give a rate of convergence of $|\mathbf{u}^1-\bar{\mathbf{u}}|$ (see \cite{CDI, Hung-book} for details). Then we can use our iteration argument to get a rate of convergence of $|\mathbf{u}^\ep-\bar{\mathbf{u}}|$.

Moreover, it is rather important to pay attention that one can never expect a way to find $v_i(\cdot,p,\bar{\mathbf{u}}(x,t))$, where we denote by $v_i(x,\cdot,p,\bar{\mathbf{u}}(x,t))$ a solution to \eqref{eq:cell} with $\mathbf{c}=\bar{\mathbf{u}}(x,t)$, such that
\begin{equation}\label{lip-cell}
p\mapsto v_i(x,\cdot,p,\bar{\mathbf{u}}(x,t))
\quad\text{is Lipschitz continuous}.
\end{equation}
In general, we have no even continuous selection of the cell problem for the first-order Hamilton--Jacobi equation.
See \cite[Section 5]{MTY}.
If we were in a situation to have \eqref{lip-cell}, then one can obtain the rate $O(\ep^{1/2})$ for $|\mathbf{u}^1-\bar{\mathbf{u}}|$ in a general setting (see \cite[Section 4.7]{Hung-book} for details). Then by applying our iteration argument, one can obtain the rate $O(\ep^{1/2})$ for $|\mathbf{u}^\ep-\bar{\mathbf{u}}|$, which is exactly the rate provided in Theorem \ref{thm1} above.

\medskip
The following example shows the optimality of \eqref{r2}. This example is inspired by \cite[Proposition 4.1]{HJ}. For $t>\sqrt{\varepsilon}$, the optimality is unclear even for single equations (see \cite[Section 4]{HJ}).
\begin{example}\label{ex}
Let $n=1$, and consider
\begin{equation}\label{ex-1}
  \left\{
   \begin{aligned}
   &\partial_t u^\ep_1(x,t)+\frac{1}{2}(\partial_x u^\ep_1(x,t))^2-V(x/\varepsilon)+F(u^\ep_1(x,t),u^\ep_2(x,t))=0
   \quad\textrm{for} \ x\in\R, \ t\in(0,\infty), \\
   &\partial_t u^\ep_2(x,t)+(\partial_x u^\ep_2(x,t))^2+F(u^\ep_1(x,t),u^\ep_2(x,t))=0
   \quad\textrm{for} \ x\in\R, \ t\in(0,\infty), \\
   &u^\ep_1(x,0)=u^\ep_2(x,0)=0
   \quad\textrm{for} \ x\in\R,
   \end{aligned}
   \right.
\end{equation}
where $V:\T\to\R$, $F:\R^2\to\R$ are given continuous functions satisfying
\[
\min_{y\in\mathbb T^1}V(y)=0,\quad V\geqslant 1\quad \textrm{on}\quad [-1/3,1/3],\quad F(\mathbf{0})=0.
\]
Here, we note that since we consider one dimensional case, we denote by $\partial_x$ for the partial derivative with respect to $x\in\mathbb R$ instead of $D$ for the spatial gradient.
Then,  for $\varepsilon\in (0,1)$ and $t>0$, we have $\bar{\mathbf{u}}(x,t)=\mathbf{0}$ and
\[
\|\mathbf{u}^\varepsilon-\bar{\mathbf{u}}\|_{L^\infty}\geqslant \min\{t/2,\varepsilon/6\}
\quad\text{for} \ x\in\R.
\]
\end{example}

Next, we consider the stationary problem of the form:
\begin{equation}\label{hje}\tag{E$^\lambda_\varepsilon$}
  \lambda u^\ep_i(x)+H_i\Big(x,\frac{x}{\varepsilon},Du^\ep_i(x),\mathbf{u}^\ep(x)\Big)=0\quad \textrm{for} \ x\in\mathbb R^n, \ i\in\{1,\dots,m\}.
\end{equation}
For related discussion on the above system, one can refer to \cite{IJ}. Here, we give the rate of convergence of the solution of the system \eqref{hje} to the system
\begin{equation}\label{hj0}\tag{$\bar{\textrm{E}}_\lambda$}
  \lambda \bar{u}_i(x)+\bar H_i(x,D\bar{u}_i(x),\bar{\mathbf{u}}(x))=0\quad \textrm{for}\ x\in \mathbb R^n,\quad i\in\{1,\dots,m\}.
\end{equation}
The second main result in this paper is as follows. One can refer to \cite[Theorem 1.2]{HJ} for a result similar to Theorem \ref{thm2} for single equations.
\begin{result}\label{thm2}
Assume {\rm(H1)-(H5)}, and
\begin{itemize}
\item [{\rm(H6)}] for all $\mathbf{c}=(c,\ldots,c)\in \mathbb R^m$, we have $H_i(x,y,p,\mathbf{u}+\mathbf{c})=H_i(x,y,p,\mathbf{u})$ for all $(x,y,p,\mathbf{u})\in\mathbb R^n\times\mathbb T^n\times\mathbb R^n\times\R^m$.
\end{itemize}
Let $\lambda>\Theta$, where $\Theta$ is a positive constant given by {\rm(H5)}. Then, there exist unique solutions $\mathbf{u}^\varepsilon$ and $\bar{\mathbf{u}}$ of \eqref{hje} and \eqref{hj0}, respectively. There is a constant $C$ depending only on $n$, $\mathbf{H}$ and $\lambda$ such that for all $\varepsilon\in(0,\frac{1}{\lambda^2})$ we have
\[
|\mathbf{u}^\varepsilon(x)-\bar{\mathbf{u}}(x)|\leqslant C\sqrt{\varepsilon}
\quad\text{for} \ x\in\R^n.
\]
\end{result}
A typical case of (H6) is the linear coupling case
\[H_i(x,y,p,\mathbf{u})=h_i(x,y,p)+\sum_{j=1}^m b_{ij}u_j,\]
where $h_i:\mathbb R^n\times\mathbb T^n\times\mathbb R^n\to\mathbb R$ is a given function, and $b_{ij}\in\mathbb R$ are given constants satisfying $\sum\limits_{j=1}^mb_{ij}=0$. See \cite{DSZ,DZ} for example. Here, we note that if we do not require that $b_{ij}\leqslant 0$ for $i,j\in\{1,\ldots,m\}$ and $i\neq j$, the coupling matrix $(b_{ij})$ is not monotone (see \cite{IJ}). Then one can not expect that \eqref{hje} is strictly monotone for all $\lambda>0$. 
A more general example satisfying (H5), (H6) is 
\[H_i(x,y,p,\mathbf{u})=h_i(x,y,p)+F\Big(\sum_{j=1}^m b_{ij}u_j\Big),\]
where $F:\mathbb R\to\mathbb R$ is a Lipschitz function. 
In fact, (H6) does not imply (H5). A simple example is
\[
H_i(x,y,p,\mathbf{u})=h_i(x,y,p)+\Big(\sum_{j=1}^m b_{ij}u_j\Big)^2,
\]
where $\sum\limits_{j=1}^mb_{ij}=0$. 
The above function satisfies (H6) but does not satisfy (H5).
\begin{remark}
{\rm
In Theorem \ref{thm2}, we require $\lambda>\Theta$ to get the uniqueness of the solution of \eqref{hje} and \eqref{hj0} (see Lemma \ref{cp}). In addition, we use $\lambda>\Theta$ to prove the convergence of the sequence $\{\mathbf{u}^\ell\}_{\ell\in\mathbb N}$ in Lemma \ref{unx} below.
}
\end{remark}

\bigskip
This paper is organized as follows.
In Section \ref{sec:pre} we give basic properties of viscosity solutions of \eqref{E} and \eqref{E0}.
Section \ref{sec:thm1} is devoted to the proof of Theorem \ref{thm1}.
In Section \ref{sec:thm2} we give a proof of Theorem \ref{thm2}.

\section{Preliminaries}\label{sec:pre}
In this section we give some of basic results of solutions of \eqref{E} and \eqref{E0}.

\begin{lemma}\label{barH}
Assume  {\rm(H1)-(H5)}.   We have
\begin{itemize}
\item [{\rm(a)}] $\lim_{|p|\to+\infty}\inf_{x\in\mathbb R^n}\bar H_i(x,p,\mathbf{0})=+\infty$.
\item [{\rm(b)}] $p\mapsto \bar H_i(x,p,\mathbf{u})$ is convex for all $(x,\mathbf{u})\in \mathbb R^n \times\mathbb R^m$.
\item [{\rm(c)}]
$\bar H_i\in\BUC(\R^n\times B(0,R)\times\R^m)$ for each $R>0$, and
$|\bar H_i(x_1,p,\mathbf{u})-\bar H_i(x_2,p,\mathbf{u})|\leqslant Lip(H)|x_1-x_2|$ for any $x_1,x_2\in\mathbb R^n$ and $(p,\mathbf{u})\in \mathbb R^n\times\mathbb R^m$.
\item [{\rm(d)}] $|\bar H_i(x,p,\mathbf{u}^1)-\bar H_i(x,p,\mathbf{u}^2)|\leqslant \Theta |\mathbf{u}^1-\mathbf{u}^2|$ for any $\mathbf{u}^1,\mathbf{u}^2\in \mathbb R^m$ and $(x,p)\in \mathbb R^n\times\mathbb R^n$.
\end{itemize}
\end{lemma}
\begin{proof}
For the proofs for (a)--(c), see \cite[Lemma 2.2]{Evans} for instance,
which are standard.
Here, we only prove (d).
Let $\delta>0$, and $\mathbf{w}^{\delta,k}$ for $k=1,2$ be solutions, respectively, to
\[
\delta w^{\delta,1}_i(y)+H_i(x,y,p+Dw^{\delta,1}_i(y),\mathbf{u}^1)=0\quad \textrm{for}\ y\in\mathbb T^n, \ i\in\{1,\dots,m\}
\]
and
\[
\delta w^{\delta,2}_i(y)+H_i(x,y,p+Dw^{\delta,2}_i(y),\mathbf{u}^2)=0\quad \textrm{for}\ y\in\mathbb T^n, \ i\in\{1,\dots,m\}.
\]
Since
\[H_i(x,y,p+Dw^{\delta,1}_i(y),\mathbf{u}^1)\leqslant H_i(x,y,p+Dw^{\delta,1}_i(y),\mathbf{u}^2)+\Theta |\mathbf{u}^1-\mathbf{u}^2|,\]
by the comparison principle, we have
\[\delta w^{\delta,1}_i-\Theta |\mathbf{u}^1-\mathbf{u}^2|\leqslant \delta w^{\delta,2}_i.\]
Sending $\delta\to 0$ yields
\[-\delta w^{\delta,1}_i\to \bar H_i(x,p,\mathbf{u}^1),\quad -\delta w^{\delta,2}_i\to \bar H_i(x,p,\mathbf{u}^2),\]
which implies that
\[\bar H_i(x,p,\mathbf{u}^2)-\bar H_i(x,p,\mathbf{u}^1)\leqslant \Theta |\mathbf{u}^1-\mathbf{u}^2|.\]
By symmetry, we get the other inequality.
\end{proof}

\begin{lemma}
For each $\varepsilon>0$, there are unique solutions of \eqref{E} and \eqref{E0}.
\end{lemma}
\begin{proof}
We only prove the existence and uniqueness of solutions of \eqref{E0}. Here we note that, once we prove the Lipschitz continuity of a solution $\mathbf{u}^\varepsilon$ of \eqref{E}, the regularity of $H_i$ in $y$ required by the comparison principle can be reduced to the continuity of $H_i$, see \cite[Theorem 2]{IJ} and \cite[Lemma B.1]{N}. Similar to the proof of Lemma \ref{lipu} below, we can prove the existence of a Lipschitz continuous solution $\mathbf{u}^\varepsilon$ of \eqref{E}. Let $\lambda>\Theta$. Setting
\begin{equation}\label{H-lam}
\bar H^\lambda_i(x,t,p,\mathbf{u}):=\lambda u_i(x,t)+e^{-\lambda t}\bar H_i(x,e^{\lambda t}p,e^{\lambda t}\mathbf{u}),
\end{equation}
we consider
\begin{equation}\label{E0'}
  \left\{
   \begin{aligned}
   &\partial_t v_i(x,t)+\bar H^\lambda_i(x,t,Dv_i(x,t),\mathbf{v}(x,t))=0\quad \textrm{for} \ (x,t)\in\mathbb R^n\times(0,T), \ i\in\{1,\dots,m\},
   \\
   &v_i(x,0)=\varphi_i(x) \quad \text{for} \ x\in\R, \ i\in\{1,\dots,m\}.
   \end{aligned}
   \right.
\end{equation}
Letting $u^1_\ell-u^2_\ell=\max_{i\in\{1,\ldots,m\}}|u^1_i-u^2_i|$,
we see that
\begin{align*}
&\bar H^\lambda_\ell(x,t,p,\mathbf{u}^1)-\bar H^\lambda_\ell(x,t,p,\mathbf{u}^2)
\\= &\lambda(u^1_\ell-u^2_\ell)+e^{-\lambda t}(\bar H_i(x,e^{\lambda t}p,e^{\lambda t}\mathbf{u}^1)-\bar H_i(x,e^{\lambda t}p,e^{\lambda t}\mathbf{u}^2))
\\ \geqslant &\lambda(u^1_\ell-u^2_\ell)-\Theta\max_{i\in\{1,\ldots,m\}}|u^1_i-u^2_i|=(\lambda-\Theta)(u^1_\ell-u^2_\ell),
\end{align*}
which implies that \eqref{E0'} is a monotone system.
Thus, by the standard theory of viscosity solutions (see \cite{CLL,IK} for instance), we obtain the unique solution $\bar{\mathbf{v}}$ to \eqref{E0'}.
Setting $\bar u_i=e^{\lambda t}\bar{v}_i$ for $i\in\{1,\dots,m\}$,
we get the conclusion.
\end{proof}

\begin{lemma}\label{lipu}
Let $\bar{\mathbf{u}}$ be the solution to \eqref{E0}.
There is a positive constant $C$ which depends on $\|\varphi_i\|_{W^{1,\infty}}$, $\bar H_i$ and $T>0$ such that
$\|\partial_t \bar u_i\|_{L^\infty}+\|D\bar u_i\|_{L^\infty}\leqslant C$ for all $i\in\{1,\ldots,m\}$.
\end{lemma}
\begin{proof}
We first consider the case where $\boldsymbol{\varphi}\in C^1(\mathbb{R}^n;\mathbb R^m)\cap\text{BUC}(\mathbb R^n;\mathbb R^m)$. Define
\[C_\varphi=\max_{i\in\{1,\dots,m\}}\sup_{x\in\mathbb R^n} |\bar H_i(x,D\varphi_i(x),\boldsymbol{\varphi}(x))|.\]
Let $\Theta>0$ be the constant given by (H5). 
We denote by $\mathcal{K}$ the set of all functions 
$\textbf{v}\in \text{Lip}(\mathbb R^n\times[0,T];\mathbb R^m)$ 
satisfying
\[|v_i(x,t)-\varphi_i(x)|\leq \frac{e^{\Theta t}-1}{\Theta}C_\varphi,\]
and 
\[
|\partial_t v_i(x,t)|\leq C_\varphi e^{\Theta t},\quad |Dv_i(x,t)|\leq R
\quad
\text{for} \ a.e. \ (x,t)\in\mathbb R^n\times(0,T), 
\]
where $R>0$ is a constant so that $\|D\varphi_i\|_{L^\infty}\leq R$ and
\[
\bar H_i(x,p,0)>C_\varphi e^{\Theta T}+(e^{\Theta T}-1)C_\varphi+\Theta\|\boldsymbol{\varphi}\|_{L^\infty}
\quad\text{for all} \ x\in\mathbb{R}^n \ \text{and} \ p\in\mathbb{R}^n\setminus B(0,R). 
\]

We define the map $F:\mathcal{K}\to C(\mathbb R^n\times[0,T];\mathbb R^m)$ by 
$\textbf{u}=(u_1,\ldots,u_m):=F(\textbf{v})$, 
where $u_i$ is the viscosity solution constructed by the Perron method of
\begin{equation}\label{ev}
\begin{cases}
\partial_t u_i(x,t)+\bar H_i(x,Du_i(x,t),\textbf{v}(x,t))=0 
\quad\text{for all} \ (x,t)\in\mathbb{R}^n\times(0,T), 
\\ u_i(x,0)=\varphi_i(x) 
\quad\text{for all} \ x\in\mathbb{R}^n 
\end{cases}
\end{equation}
for $i\in\{1,\ldots,m\}$. By the Perron method, we have 
\begin{align*}
&u_i(x,t)\\
=&\, 
\sup\Big\{w\in\Lip(\mathbb{R}^n\times[0,T])\mid 
w \ \text{is a subsolution to} \ \eqref{ev} \ \text{satisfying} \\ 
&
\hspace*{3.5cm}
w\le\varphi_i(x)+\sup_{(x,t)\in\mathbb R^n\times(0,T)}|\bar H_i(x,D\varphi_i(x),\textbf{v}(x,t))|t
\ \text{on} \ \mathbb{R}^n\times[0,T]
\Big\}. 
\end{align*}
In the following, we aim to show that $F$ is from $\mathcal K$ to itself. Note that \eqref{ev} is a decoupled system of Hamilton-Jacobi equations. According to \cite[Theorem 5.1]{B}, if $w_i(x,t)$ (resp., $v_i(x,t)$) is a viscosity supersolution (resp., subsolution) of \eqref{ev}, and either $w_i$ or $v_i$ is Lipschitz continuous in $x$ uniformly in $t$, then we have $w_i\geq v_i$ on $\mathbb{R}^n\times[0,T]$.

We define $u^1_i(x,t):=\varphi_i(x)-\frac{e^{\Theta t}-1}{\Theta}C_\varphi$. Since $\mathbf{v}\in\mathcal K$, we have
\begin{align*}
&\partial_t u^1_i(x,t)+\bar H_i(x,Du^1_i(x,t),\mathbf{v}(x,t))
\\ &\leq -e^{\Theta t}C_\varphi+\bar H_i(x,D\varphi_i(x),\boldsymbol{\varphi}(x))+\Theta |\mathbf{v}(x,t)-\boldsymbol{\varphi}(x)|
\\ &\leq \bar H_i(x,D\varphi_i(x),\boldsymbol{\varphi}(x))-C_\varphi\leq 0,
\end{align*}
which implies that $u^1_i$ is a classcial subsolution of \eqref{ev}. Similarly, we define $u^2_i(x,t):=\varphi_i(x)+\frac{e^{\Theta t}-1}{\Theta}C_\varphi$. One can check that it is a classical supersolution of \eqref{ev}. Since $u^1_i$ and $u^2_i$ are Lipschitz continuous in $x$, by the comparison principle we have
\[
|u_i(x,t)-\varphi_i(x)|\leq \frac{e^{\Theta t}-1}{\Theta}C_\varphi.
\]

For each $h>0$, define
\begin{equation*}
\tilde u_i(x,t)=
\begin{cases}
\varphi_i(x)-C_\varphi e^{\Theta h}t\quad &\text{for} \ t\leq h, \ x\in\mathbb{R}^n, 
\\
u_i(x,t-h)-C_\varphi he^{\Theta t}\quad &\text{for} \ t>h, x\in\mathbb{R}^n. 
\end{cases}
\end{equation*}
Then for $t\leq h$, we have
\begin{align*}
&\partial_t \tilde{u}_i(x,t)+\bar H_i(x,D\tilde{u}_i(x,t),\mathbf{v}(x,t))
\\&\leq -C_\varphi e^{\Theta h}+\bar H_i(x,D\varphi_i(x),\boldsymbol{\varphi}(x))+\Theta |\mathbf{v}(x,t)-\boldsymbol{\varphi}(x)|
\\ &\leq -C_\varphi e^{\Theta h}+\bar H_i(x,D\varphi_i(x),\boldsymbol{\varphi}(x))+C_\varphi e^{\Theta t}-C_\varphi \leq 0.
\end{align*}
For $t>h$, in the viscosity sense we have
\begin{align*}
\partial_t \tilde u_i(x,t)&=\partial_t u_i(x,t-h)-C_\varphi h \Theta e^{\Theta t}
\\ &=-\bar H_i(x,Du_i(x,t-h),\mathbf{v}(x,t-h))-C_\varphi h \Theta e^{\Theta t}
\\ &\leq -\bar H_i(x,Du_i(x,t-h),\mathbf{v}(x,t))+\Theta |\mathbf{v}(x,t)-\mathbf{v}(x,t-h)|-C_\varphi h \Theta e^{\Theta t}
\\ &\leq -\bar H_i(x,Du_i(x,t-h),\mathbf{v}(x,t))=\bar H_i(x,D\tilde u_i(x,t),\mathbf{v}(x,t)).
\end{align*}
Therefore, $\tilde u_i$ is a subsolution of \eqref{ev} almost everywhere. By the convexity of $\bar H_i$, it is also a viscosity subsolution of \eqref{ev} satisfying $\tilde u_i(x,0)=\varphi_i(x)$. Since $u$ is constructed by the Perron method, we get
\[u_i(x,t)\geq u_i(x,t-h)-C_\varphi he^{\Theta t},\]
which implies that
\[\partial_t u_i\geq -C_\varphi e^{\Theta t}.\]
Then we have
\[\bar H_i(x,Du_i,0)\leq C_\varphi e^{\Theta T}+(e^{\Theta T}-1)C_\varphi+\Theta\|\boldsymbol{\varphi}\|_{L^\infty},\]
which implies that $|Du_i|\leq R$. Define
\begin{equation*}
\tilde u_i(x,t)=
\begin{cases}
\varphi_i(x)+C_\varphi e^{\Theta h}t\quad &\text{for} \ t\leq h, x\in\mathbb{R}^n, 
\\
u_i(x,t-h)+C_\varphi he^{\Theta t}\quad & \text{for} \ t>h, x\in\mathbb{R}^n. 
\end{cases}
\end{equation*}
One can check that $\tilde u_i$ is a supersolution of \eqref{ev} for $t<h$ and $t>h$. By the comparison principle, we get $\tilde u_i(x,h)\geq u_i(x,h)$. Since $u_i$ is Lipschitz continuous in $x$ and $\tilde u_i$ is a  supersolution of \eqref{ev} for $t>h$, we get $\tilde u(x,t)\geq u_i(x,t)$ for $t>h$, that is,
\[u_i(x,t)\leq u_i(x,t-h)+C_\varphi he^{\Theta t},\]
which implies that
\[\partial_tu_i\leq C_\varphi e^{\Theta t}.\]
This shows that $F$ is from $\mathcal K$ to itself.

It is clear that $\boldsymbol{\varphi}\in\mathcal{K}$. We define $\mathbf{u}^\ell:=F^\ell(\boldsymbol{\varphi})$ for $\ell\in\N$. Then using the optimal control formula of $u^\ell_i(x,t)$, we can prove
\[|\mathbf{u}^{\ell+1}(x,t)-\mathbf{u}^\ell(x,t)|\leq \frac{(\Theta t)^\ell}{\ell !}\|\mathbf{u}^1-\boldsymbol{\varphi}\|_{L^\infty}\]
by induction. We skip the proof since it is quite similar to that of  Lemma \ref{un}. Since $\mathbf{u}^1\in\mathcal{K}$, $\|\mathbf{u}^1-\boldsymbol{\varphi}\|_{L^\infty}$ is bounded. Then for all $m>\ell$, we have
\[\|\mathbf{u}^m-\mathbf{u}^\ell\|_{L^\infty}\leq \sum_{k=\ell}^{m-1}\|\mathbf{u}^{k+1}-\mathbf{u}^k\|_{L^\infty}\leq  \|\mathbf{u}^{1}-\boldsymbol{\varphi}\|_{L^\infty}\sum_{k=\ell}^{m-1}\frac{(\Theta T)^k}{k!}\to 0\quad \text{as}\quad\ell\to\infty.\]
Thus, the sequence $\{\mathbf{u}^\ell\}_{\ell\in\N}$ is a Cauchy sequence in $(C_b(\mathbb R^n\times[0,T]),\|\cdot\|_{L^\infty})$, where $C_b$ stands for the space of all bounded and continuous functions. As a consequence, $\mathbf{u}^\ell$ converges to a continuous function $\mathbf{u}^\infty$ uniformly on $\R^n\times[0,T]$. By the stability of viscosity solutions, $\mathbf{u}^\infty$ is a solution of \eqref{E0}. Since $\mathcal K$ is closed, $\bar{\mathbf{u}}\in \mathcal{K}$. Thus, we conclude that $\bar{\textbf{u}}$ has the Lipschitz constant depending on $T$.

We finish the proof by an approximation. For $\boldsymbol{\varphi}\in \text{Lip}(\mathbb R^n;\mathbb R^m)\cap\text{BUC}(\mathbb R^n;\mathbb R^m)$, we take a sequence $\{\boldsymbol{\varphi}^k\}\subset C^1(\R^n;\R^m)$ with $\|\varphi^k_i-\varphi_i\|_{W^{1,\infty}}\to 0$. Let $\bar{\mathbf{u}}^k$ be the solution of \eqref{E0} with $\bar{u}^k_i(x,0)=\varphi^k_i(x)$. According to the argument above, $\{\bar{\mathbf{u}}^k\}$ is uniformly bounded and equi-Lipschitz continuous. Up to a subsequence, $\{\bar{\mathbf{u}}^k\}$ locally uniformly converges to a solution $\bar{\textbf{u}}$ of \eqref{E0}. Then $\bar{\textbf{u}}$ has the Lipschitz constant depending on $T$.
\end{proof}

\begin{remark}
{\rm
When the coupling of $\bar{\mathbf{H}}$ is monotone, one can easily check that $\hat v_i(x,t)=\varphi_i(x)+C_\varphi t$ (resp., $\check v_i(x,t)=\varphi_i(x)-C_\varphi t$) is a supersolution (resp., subsolution) of \eqref{E0}. Then by the comparison principle one gets a time global estimate $|\partial_t \bar u_i|+|D\bar u_i|\leqslant C$, where $C>0$ is independent of $T$.}
\end{remark}

\section{Proof of Theorem {\rm\ref{thm1}}}\label{sec:thm1}
In this section we give a proof of Theorem \ref{thm1}. 
Similar to Lemma \ref{lipu}, we can prove that there is a positive constant $C$ which depends on $\|\varphi_i\|_{W^{1,\infty}}$, $H_i$ and $T>0$ such that
$\|\partial_t u^\ep_i\|_{L^\infty}+\|Du^\ep_i\|_{L^\infty}\leqslant C$ for all $i\in\{1,\ldots,m\}$. Then we can modify $\textbf{H}$ and $\bar{\textbf{H}}$ to be superlinear for $|p|>C$. Then the Lagrangians associated with $H_i$ and $\bar H_i$ are finite valued. In this section, we assume that both $H_i$ and $\bar H_i$ are superlinear. 

Let $\bar{u}$ be the solution of \eqref{E0}, and then it is rather clear to see that we have the implicit representation formula
\begin{equation}\label{def:u-bar}
\bar u_i(x,t)=\inf_{\gamma\in\cC(x;t)}\bigg\{\varphi_i(\gamma(0))+\int_0^t \bar L_i(\gamma(s),\dot\gamma(s),\bar{\mathbf{u}}(\gamma(s),s))\, ds\bigg\},
\end{equation}
where we set
$\bar L_i(x,v,\mathbf{u}):=\sup_{p\in\mathbb R^n}(v\cdot p-\bar H_i(x,p,\mathbf{u}))$ for $i\in\{1,\ldots,m\}$, and
$\cC(x;t)$ denotes the set of all trajectories $\gamma\in \AC([0,t])$
such that $\gamma(t)=x$.
Here, we denote by $\AC([a,b])$ for $-\infty\le a\le b\le\infty$ the set of absolutely continuous functions on $[a,b]$ with values in $\R^n$.


According to Lemma \ref{lipu}, it is standard to show the following property.
\begin{lemma}\label{lipgm}
The infimum of \eqref{def:u-bar} is achieved, and let $\bar\gamma\in\cC(x;t)$
be a minimizer.
There is a constant $M_0$ depending on $\bar{\mathbf{H}}$, $\|\boldsymbol{\varphi}\|_{W^{1,\infty}}$ and $T$ such that $\|\dot{\bar{\gamma}}\|_{L^\infty}\leqslant M_0$.
\end{lemma}

We now define the function $u^1_i:\R^n\times[0,T]\to\R$ by
\begin{equation}\label{def:u1}
u^1_i(x,t)=\inf_{\gamma\in\cC(x;t)}\bigg\{\varphi_i(\gamma(0))+\int_0^t L_i\Big(\gamma(s),\frac{\gamma(s)}{\varepsilon},\dot\gamma(s),\bar{\mathbf{u}}(\gamma(s),s)\Big)\, ds\bigg\}
\end{equation}
for $i\in\{1,\ldots,m\}$, where we set
$
L_i(x,y,v,\mathbf{u}):=\sup_{p\in\mathbb R^n}(v\cdot p-H_i(x,y,p,\mathbf{u})).
$
It is rather standard to see that $\mathbf{u}^1$ is continuous on $\R^n\times[0,\infty)$, and is the solution to
\begin{equation}\label{u1eq}
  \left\{
   \begin{aligned}
   &\partial_t u_i(x,t)+H_i\Big(x,\frac{x}{\varepsilon},Du_i(x,t),\bar{\mathbf{u}}(x,t)\Big)=0\quad \textrm{for} \ (x,t)\in\mathbb R^n\times(0,T), \ i\in\{1,\dots,m\},
   \\
   &u_i(x,0)=\varphi_i(x)\quad \text{for} \ x\in\R^n, \  i\in\{1,\dots,m\}.
   \\
   \end{aligned}
   \right.
\end{equation}

\begin{lemma}\label{lipgm1}
The infimum of \eqref{def:u1} is achieved, and let $\gamma_1\in\cC(x;t)$
be a minimizer.
There is a constant $M_0$ depending on $\mathbf{H}$, 
$\|\boldsymbol{\varphi}\|_{W^{1,\infty}}$ and $T$ such that $\|\dot{\gamma_1}\|_{L^\infty}\leqslant M_0$.
\end{lemma}

For $c,x,y\in\mathbb R^n$, $d\in\mathbb R^m$, $\varepsilon>0$, 
$i\in\{1,\ldots,m\}$ and $0\leqslant t_1\leqslant t_2<+\infty$, we set
\[
m_i^\varepsilon(t_1,t_2,x,y;c,d):=
\inf
\left\{\int_{t_1}^{t_2}L_i\Big(c,\frac{\gamma(s)}{\varepsilon},\dot\gamma(s),d\Big)\,ds\mid \gamma\in\cC(x,y;t_1,t_2)
\right\},
\]
and
\[
\bar m_i(t_1,t_2,x,y;c,d):=
\inf\left\{\int_{t_1}^{t_2}\bar L_i(c,\dot\gamma(s),d)\, ds\mid\gamma\in\cC(x,y;t_1,t_2)
\right\},
\]
where we denote by $\cC(x,y;a,b)$ for $x,y\in\R^n$, $a,b\in\R$
the set of all trajectories $\gamma\in \AC([0,t])$
satisfying $\gamma(a)=x$, $\gamma(b)=y$.

The next lemma is a direct result of \cite[Lemma 1.1]{HJ}.
\begin{lemma}\label{lem:HJ}
Fix $c\in\R^n$, $d\in\R^m$.
Let $x,y\in\mathbb R^n$, $\varepsilon,t>0$ and $M_0$ with $|y-x|\leqslant M_0t$.
There exists a constant $C>0$ depending on $n$, $M_0$, $\mathbf{H}$ and $\|\boldsymbol{\varphi}\|_{W^{1,\infty}}$ such that
\[
|m_i^\varepsilon(0,t,x,y;c,d)-\bar m_i(0,t,x,y;c,d)|\leqslant C\varepsilon
\]
for all $i\in\{1,\ldots, m\}$. 
\end{lemma}

\begin{lemma}\label{u1}
There is a constant $C>0$ depending only on $n$, $\mathbf{H}$, 
$\|\boldsymbol{\varphi}\|_{W^{1,\infty}}$ and $T$ such that for $t\in[\sqrt{\varepsilon},T]$ we have
\[
|\mathbf{u}^1(x,t)-\bar{\mathbf{u}}(x,t)|\leqslant C\sqrt{\varepsilon}
\quad\text{for all} \ x\in\R, \ \ep>0,
\]
and for $t\in(0,\sqrt{\varepsilon})$ we have
\[
|\mathbf{u}^1(x,t)-\bar{\mathbf{u}}(x,t)|\leqslant C \varepsilon
\quad\text{for all} \ x\in\R, \ \ep>0.
\]
\end{lemma}
\begin{proof}
We only prove that
\[
u^1_i(x,t)-\bar u_i(x,t)\geqslant -C\sqrt{\varepsilon}
\quad \text{for all} \ t\in[\sqrt{\varepsilon},T],
\]
and
\[
u^1_i(x,t)-\bar u_i(x,t)\geqslant -C \varepsilon
\quad \text{for all} \ t\in(0,\sqrt{\varepsilon}),
\]
as the other inequalities are similarly obtained.
We consider a partition of $[0,t]$:
\[0=t_0\leqslant t_1\leqslant t_2\leqslant \dots\leqslant t_k\leqslant \dots\leqslant t_N\leqslant t_{N+1}=t.\]
For each $k=0,\dots,N$, we have
\begin{equation}\label{1}
\begin{aligned}
&\int_{t_k}^{t_{k+1}}L_i\left(\gamma_1(s),\frac{\gamma_1(s)}{\varepsilon},\dot\gamma_1(s),\bar{\mathbf{u}}(\gamma_1(s),s)\right)\, ds
\\ \geqslant\,
&\int_{t_k}^{t_{k+1}}L_i\left(\gamma_1(t_k),\frac{\gamma_1(s)}{\varepsilon},\dot\gamma_1(s),\bar{\mathbf{u}}(\gamma_1(t_k),t_k)\right)\, ds-E_k
\\ \geqslant\,
& m_i^\varepsilon(t_k,t_{k+1},\gamma_1(t_k),\gamma_1(t_{k+1});\gamma_1(t_k),\bar{\mathbf{u}}(\gamma_1(t_k),t_k))-E_k
\\ \geqslant\,
& \bar m_i(t_k,t_{k+1},\gamma_1(t_k),\gamma_1(t_{k+1});\gamma_1(t_k),\bar{\mathbf{u}}(\gamma_1(t_k),t_k))-C\varepsilon-E_k,
\end{aligned}
\end{equation}
where we set
\[
E_k:=\int_{t_k}^{t_{k+1}}\bigg[\textrm{Lip}(H)|\gamma_1(s)-\gamma_1(t_k)|+\Theta\Big(|\partial_t\bar{\mathbf{u}}|\cdot|s-t_k|+|D\bar{\mathbf{u}}|\cdot|\gamma_1(s)-\gamma_1(t_k)|\Big)\bigg]\, ds,
\]
and we have used Lemma \ref{lem:HJ} in the third inequality. 
Here, we use the notations 
$|\partial_t\bar{\textbf{u}}|:=\max_{i\in\{1,\dots,m\}}|\partial_t u_i|$ and $|D\bar{\textbf{u}}|:=\max_{i\in\{1,\dots,m\}}|Du_i|$. 

Now, take $\gamma^k\in\AC([t_k,t_{k+1}])$ with $\gamma^k(t_k)=\gamma_1(t_k)$ and $\gamma^k(t_{k+1})=\gamma_1(t_{k+1})$ be a curve satisfying
\[
\int_{t_k}^{t_{k+1}} \bar L_i(\gamma_1(t_k),\dot\gamma^k(s),\bar{\mathbf{u}}(\gamma_1(t_k),t_k))\, ds
= \bar m_i(t_k,t_{k+1},\gamma_1(t_k),\gamma_1(t_{k+1});\gamma_1(t_k),\bar{\mathbf{u}}(\gamma_1(t_k),t_k)).
\]
Then,
\begin{equation}\label{2}
\bar m_i(t_k,t_{k+1},\gamma_1(t_k),\gamma_1(t_{k+1});\gamma_1(t_k),\bar{\mathbf{u}}(\gamma_1(t_k),t_k))
\geqslant
\int_{t_k}^{t_{k+1}} \bar L_i(\gamma^k(s),\dot\gamma^k(s),\bar{\mathbf{u}}(\gamma^k(s),s))\, ds-\bar E_k,
\end{equation}
where we set
\[
\bar E_k:=
\int_{t_k}^{t_{k+1}}\bigg[\textrm{Lip}(H)|\gamma^k(s)-\gamma_1(t_k)|+\Theta\Big(|\partial_t\bar{\mathbf{u}}|\cdot|s-t_k|+|D\bar{\mathbf{u}}|\cdot|\gamma^k(s)-\gamma_1(t_k)|\Big)\bigg]\, ds.
\]

By Lemmas \ref{lipu} and \ref{lipgm1}, we obtain
\[|E_k|\leqslant \int_{t_k}^{t_{k+1}}\Big(\textrm{Lip}(H)M_0(s-t_k)+\Theta CM_0(s-t_k)\Big)\, ds\leqslant K(t_{k+1}-t_k)^2,\]
where
\[K:=\frac{1}{2}(\textrm{Lip}(H)M_0+\Theta CM_0).\]
Note that $|\gamma^k(t_{k+1})-\gamma^k(t_k)|\leq M_0(t_{k+1}-t_k)$, 
which immediately implies 
the boundedness of $|\dot\gamma^k|$.
Similarly, there is $K>0$ such  that $\bar E_k\leqslant K(t_{k+1}-t_k)^2$.

Now, combining \eqref{1} and \eqref{2}, we obtain
\begin{align*}
u^1_i(x,t)&=\varphi_i(\gamma_1(0))+\sum_{k=0}^N \int_{t_k}^{t_{k+1}}L_i\Big(\gamma_1(s),\frac{\gamma_1(s)}{\varepsilon},\dot\gamma_1(s),\bar{\mathbf{u}}(\gamma_1(s),s)\Big)\, ds
\\ &\geqslant \varphi_i(\gamma_1(0))
+\sum_{k=0}^N
\int_{t_k}^{t_{k+1}} \bar L_i(\gamma^k(s),\dot\gamma^k(s),\bar{\mathbf{u}}(\gamma^k(s),s))\, ds
-\sum_{k=0}^N(E_k+\bar E_k+C\varepsilon)
\\ &\geqslant \bar u_i(x,t)-\sum_{k=0}^N(E_k+\bar E_k+C\varepsilon).
\end{align*}
Set $N:=\lfloor t/\sqrt{\varepsilon}\rfloor$ and $t_k=k\sqrt{\varepsilon}$ for $k=0,\dots,N$, where we denote by $\lfloor r\rfloor$ the greatest integer less than or equal to $r\in\R$.
If $t\in(0,\sqrt{\varepsilon})$, then we have $N=0$. Thus,
\[
\bigg|\sum_{k=0}^N(E_k+\bar E_k+C\varepsilon)\bigg|=|E_0+\bar E_0+C\varepsilon|\leqslant 2Kt^2+C\varepsilon\leqslant (2K+C)\varepsilon.\]
If $t>\sqrt{\varepsilon}$, then $N+1\le \frac{t}{\sqrt{\ep}}+1\leqslant \frac{2t}{\sqrt{\ep}}$. Thus,
\[
\bigg|\sum_{k=0}^N(E_k+\bar E_k+C\varepsilon)\bigg|\leqslant 2(N+1)K\varepsilon+(N+1)C\varepsilon\leqslant (4TK+2TC)\sqrt{\varepsilon},
\]
which completes the proof.
\end{proof}

\medskip
Next, we define a family of functions $\{\mathbf{u}^\ell\}_{\ell\in\N}$ by
\begin{equation}\label{func:ul}
u^{\ell+1}_i(x,t):=
\inf_{\gamma\in\cC(x;t)}\bigg\{\varphi_i(\gamma(0))+\int_0^t L_i\Big(\gamma(s),\frac{\gamma(s)}{\varepsilon},\dot\gamma(s),\mathbf{u}^\ell(\gamma(s),s)\Big)\, ds\bigg\}
\end{equation}
for $\ell\in\N$.
Let $\gamma_{\ell+1}:[0,t]\to\mathbb R^n$ be minimizers of \eqref{func:ul}.
By a standard theory of viscosity solutions, we see
that $\mathbf{u}^{\ell+1}$ is continuous on $\R^n\times[0,\infty)$, and is the solution of
the decoupled system
\begin{equation*}
  \left\{
   \begin{aligned}
   &\partial_t u_i(x,t)+H_i\Big(x,\frac{x}{\varepsilon},Du_i(x,t),\mathbf{u}^\ell(x,t)\Big)=0\quad \textrm{for} \ (x,t)\in\mathbb R^n\times(0,T), \ i\in\{1,\dots,m\},
   \\
   &u_i(x,0)=\varphi_i(x)\quad \text{for} \ x\in\R^n, \ i\in\{1,\dots,m\}.
   \\
   \end{aligned}
   \right.
\end{equation*}

\begin{lemma}\label{un}
For all $\ell\in\N$, we have
\begin{equation}\label{n}
  |\mathbf{u}^{\ell+1}(x,t)-\mathbf{u}^\ell(x,t)|\leqslant \frac{(\Theta t)^{\ell}}{\ell!}C\sqrt{\varepsilon}\quad \text{for} \ x\in\R^n, \ t\in(\sqrt{\ep},T).
\end{equation}
Moreover,
\begin{equation}\label{n'}
  |\mathbf{u}^{\ell+1}(x,t)-\mathbf{u}^\ell(x,t)|\leqslant \frac{(\Theta t)^{\ell}}{\ell!}C \varepsilon\quad
  \text{for} \ x\in\R^n, \ t\in(0,\sqrt{\varepsilon}).
\end{equation}
\end{lemma}
\begin{proof}
We prove by induction. For $\ell=1$, we have
\begin{align*}
&u^2_i(x,t)-u^1_i(x,t)
\\ \leqslant & \int_0^t\bigg[L_i\Big(\gamma_1(s),\frac{\gamma_1(s)}{\varepsilon},\dot\gamma_1(s),\mathbf{u}^1(\gamma_1(s),s)\Big)-L_i\Big(\gamma_1(s),\frac{\gamma_1(s)}{\varepsilon},\dot\gamma_1(s),\bar{\mathbf{u}}(\gamma_1(s),s)\Big)\bigg]\, ds
\\ \leqslant &\int_0^t\Theta|\mathbf{u}^1(\gamma_1(s),s)-\bar{\mathbf{u}}(\gamma_1(s),s)|\, ds\leqslant \Theta tC\sqrt{\varepsilon}.
\end{align*}
Assume that \eqref{n} holds for $\ell=k$. For $\ell=k+1$, we have
\begin{align*}
&u^{k+2}_i(x,t)-u^{k+1}_i(x,t)
\\ \leqslant & \int_0^t\bigg[L_i\Big(\gamma_{k+1}(s),\frac{\gamma_{k+1}(s)}{\varepsilon},\dot\gamma_{k+1}(s),\mathbf{u}^{k+1}(\gamma_{k+1}(s),s)\Big)
\\ &-L_i\Big(\gamma_{k+1}(s),\frac{\gamma_{k+1}(s)}{\varepsilon},\dot\gamma_{k+1}(s),\mathbf{u}^k(\gamma_{k+1}(s),s)\Big)\bigg]\, ds
\\ \leqslant &\int_0^t\Theta|\mathbf{u}^{k+1}(\gamma_{k+1}(s),s)-\mathbf{u}^{k}(\gamma_{k+1}(s),s)|\, ds\leqslant C\sqrt{\varepsilon}\int_0^t\Theta\frac{(\Theta s)^{k}}{k!}\, ds
=\frac{(\Theta t)^{k+1}}{(k+1)!}C\sqrt{\varepsilon}.
\end{align*}
By symmetry, we can prove
\[|u^{k+2}_i(x,t)-u^{k+1}_i(x,t)|\leqslant \frac{(\Theta t)^{k+1}}{(k+1)!}C\sqrt{\varepsilon}.\]
Similarly we can prove \eqref{n'}.
\end{proof}

\begin{proof}[Proof of Theorem {\rm\ref{thm1}}]
By \eqref{n}, the sequence $\{\mathbf{u}^\ell\}_{\ell\in\N}$ is a Cauchy sequence in $C_b(\mathbb R^n\times[0,T])$, which implies that $\mathbf{u}^\ell\to \mathbf{u}^\infty$ uniformly on $\mathbb R^n\times[0,T]$ as $\ell\to\infty$. By the stability of viscosity solutions, $\mathbf{u}^\infty$ is the  solution $\mathbf{u}^\varepsilon$ of \eqref{E}. By Lemma \ref{un} we have
\begin{align*}
|\textbf{u}^\ell(x,t)-\bar{\textbf{u}}(x,t)|&\le\sum_{j=2}^{\ell}|\textbf{u}^j(x,t)-\textbf{u}^{j-1}(x,t)|+|\textbf{u}^1(x,t)-\bar{\textbf{u}}(x,t)|
\\ &\le \sum_{j=0}^{\ell-1}\frac{(\Theta t)^j}{j!}C\sqrt{\varepsilon}\le e^{\Theta T}C\sqrt{\varepsilon},
\end{align*}
which implies
\[
|\mathbf{u}^\varepsilon(x,t)-\bar{\mathbf{u}}(x,t)|\leqslant e^{\Theta T}C\sqrt{\varepsilon}
\ \text{for} \ x\in\R^n, \ t\in [\sqrt{\varepsilon},T],\]
and
\begin{equation}\label{eq:small-time}
|\mathbf{u}^\varepsilon(x,t)-\bar{\mathbf{u}}(x,t)|\leqslant e^{\Theta T}C \varepsilon
\ \text{for} \ x\in\R^n, \ t\in(0,\sqrt{\varepsilon}).
\end{equation}

By Lemma \ref{lipu}, we have
\[|\bar u_i(x,t)-\varphi_i(x)|\leqslant Ct.\]
Similarly, there is a constant $C'$ which depends on $\|\varphi_i\|_{W^{1,\infty}}$, $H_i$ and $T>0$ such that
\[|u^\varepsilon_i(x,t)-\varphi_i(x)|\leqslant C't.\]
One can conclude that
\begin{equation}\label{Ct}
  |u^\varepsilon_i(x,t)-\bar u_i(x,t)|\leqslant (C+C')t.
\end{equation}

Combining \eqref{eq:small-time} with \eqref{Ct}, we can get the conclusion.
\end{proof}


\medskip
\begin{proof}[Proof of Example {\rm\ref{ex}}]
Let $(u_1^\ep, u_2^\ep)$ be the solution to \eqref{ex-1}, and we have
\begin{align*}
&u^\varepsilon_1(0,t)=\inf_{\gamma\in\cC\left(0;\frac{t}{\ep}\right)}
\bigg\{\varepsilon\int_0^{t/\varepsilon}\bigg[\frac{1}{2}|\dot{\gamma}(s)|^2+V(\gamma(s))-F\big(u^\ep_1(\varepsilon\gamma(s),\ep s),u^\ep_2(\varepsilon\gamma(s),\ep s)\big)\bigg]\, ds\bigg\},\\
&u^\varepsilon_2(0,t)=\inf_{\gamma\in\cC\left(0;\frac{t}{\ep}\right)}\bigg\{\varepsilon\int_0^{t/\varepsilon}\bigg[\frac{1}{4}|\dot{\gamma}(s)|^2-F\big(u^\ep_1(\varepsilon\gamma(s),\ep s),u^\ep_2(\varepsilon\gamma(s),\ep s)\big)\bigg]\, ds\bigg\}.
\end{align*}
Let $\gamma:[0,t]\to\mathbb R$ be a minimizer of $u^\varepsilon_1(0,t)$.

If $\gamma_1([0,t/\varepsilon])\subset [-1/3,1/3]$, then
\[
u^\varepsilon_1(0,t)-u^\varepsilon_2(0,t)\geqslant \varepsilon\int_0^{t/\varepsilon}\bigg[\frac{1}{4}|\dot{\gamma}_1(s)|^2+V(\gamma_1(s))\bigg]\, ds
\geqslant \varepsilon\int_0^{t/\varepsilon}V(\gamma_1(s))\, ds\geqslant t.
\]
If not, without loss of generality, we assume that there exists $s_1\in(0,t/\varepsilon)$ such that $\gamma_1(s_1)=1/3$ and $\gamma((s_1,t/\varepsilon])\subset (-1/3,1/3)$. Then,
\begin{align*}
u^\varepsilon_1(0,t)-u^\varepsilon_2(0,t)&\geqslant \varepsilon\int_{s_1}^{t/\varepsilon}\bigg[\frac{1}{4}|\dot{\gamma}_1(s)|^2+V(\gamma_1(s))\bigg]\, ds
\\ &\geqslant \varepsilon\bigg(\frac{1}{4(t/\varepsilon-s_1)}\bigg|\int_{s_1}^{t/\varepsilon}\dot{\gamma}_1(s)\, ds\bigg|^2+\frac{t}{\varepsilon}-s_1\bigg)
\\ &=\varepsilon\bigg(\frac{1}{36(t/\varepsilon-s_1)}+\frac{t}{\varepsilon}-s_1\bigg)\geqslant \frac{\varepsilon}{3},
\end{align*}
where in the equality we used the fact that $\int_{s_1}^{t/\varepsilon}\dot{\gamma}_1(s)ds=\frac{1}{3}$. 
It is direct to see that
\[\frac{1}{2}(v'(y))^2-V(y)+F(\mathbf{0})=0,\quad (v'(y))^2+F(\mathbf{0})=0\]
have solutions. Therefore, the effective Hamiltonian $\bar H_i(x,0,\mathbf{0})=0$, which implies that
$\bar{\mathbf{u}}(x,t)=\mathbf{0}$ is the solution of \eqref{E0} with the initial condition $\varphi_1=\varphi_2\equiv0$.

Therefore, we finally conclude that
\begin{align*}
2\|\textbf{u}^\varepsilon-\bar{\textbf{u}}\|_{L^\infty}&\geq 2|\textbf{u}^\varepsilon(0,t)-\bar{\textbf{u}}(0,t)|=2|\textbf{u}^\varepsilon(0,t)|=2\max_{i\in\{1,\dots,m\}}|u^\varepsilon_i(0,t)|
\\ &\geq |u^\varepsilon_1(0,t)|+|u^\varepsilon_2(0,t)|\geq u^\varepsilon_1(0,t)-u^\varepsilon_2(0,t)\geq \min\{t,\varepsilon/3\}.
\end{align*} 
The proof is now complete.
\end{proof}

\section{Proof of Theorem \ref{thm2}}\label{sec:thm2}
In this section we prove Theorem \ref{thm2} by a similar argument to that of Section \ref{sec:thm1}, which needs to be carefully adjusted to be a stationary problem.
\textit{Throughout} this section, we assume $\lambda>\Theta$.

\begin{lemma}\label{cp}
Let $\mathbf{v}$ and $\mathbf{w}$ be a bounded lower semi-continuous supersolution and a bounded upper semi-continuous subsolution, respectively, of \eqref{hj0}. We have $w_i\leqslant v_i$ for all $i\in\{1,\dots,m\}$. Moreover, \eqref{hj0} has a unique solution $\bar{\mathbf{u}}$.
\end{lemma}
\begin{proof}
According to \cite{IK}, we only need to check that $\lambda \mathbf{u}+\bar{\mathbf{H}}$ is strictly monotone. Assume that $u_\ell-v_\ell=\max_{i\in1,\ldots,m}|u_i-v_i|$.
We have
\[\lambda u_\ell+\bar H_\ell(x,p,\mathbf{u})-\lambda v_\ell-\bar H_\ell(x,p,\mathbf{v})
\geqslant \lambda(u_\ell-v_\ell)-\Theta |\mathbf{u}-\mathbf{v}|=(\lambda-\Theta)(u_\ell-v_\ell),\]
which implies that $\lambda \mathbf{u}+\bar{\mathbf{H}}$ is strictly monotone when $\lambda>\Theta$.
\end{proof}

\begin{lemma}\label{bd}
Let
\[
M:=\frac{\lambda}{\lambda-\Theta}\sup_{x\in\mathbb R^n,y\in\T^n}|\mathbf{H}(x,y,0,\mathbf{0})|, \quad
\bar{M}:=\frac{\lambda}{\lambda-\Theta}\sup_{x\in\mathbb R^n}|\bar{\mathbf{H}}(x,0,\mathbf{0})|.
\]
Then,
$\|\mathbf{u}^\ep\|_{L^\infty}\le M/\lambda$ for all $\ep>0$, and
$\|\bar{\mathbf{u}}\|_{L^\infty}\leqslant \bar{M}/\lambda$.
Here, $M\geqslant \bar M$. Moreover, there is a constant $C>0$ which depends only on
$M$, $\bar{M}$, and $\lambda$ such that
\begin{equation}\label{Du(x)}
\|Du_i^\ep\|_{\Li}+  \|D\bar u_i\|_{\Li}\leqslant C
\quad\text{for all} \ i\in\{1,\ldots, m\}.
\end{equation}
\end{lemma}
\begin{proof}
We only prove for $\bar{\mathbf{u}}$. 
A direct calculation shows that
\[
\lambda \frac{\bar M}{\lambda}+\bar H_i(x,0,\frac{\mathbf{\bar M}}{\lambda})\geqslant \bar M+\bar H_i(x,0,0)-\Theta \frac{\bar M}{\lambda}=\sup_{x\in\mathbb R^n}|\bar{\mathbf{H}}(x,0,\mathbf{0})|-\bar H_i(x,0,0)\geqslant 0,
\]
where we set $\bar{\mathbf{M}}:=(\bar M,\ldots,\bar M)$.
Thus, $\bar{\mathbf{M}}/\lambda$ is a supersolution of \eqref{hj0}. Similarly, $-\bar{\mathbf{M}}/\lambda$ is a subsolution of \eqref{hj0}. Then Lemma \ref{cp} implies that $|\bar{\mathbf{u}}(x)|\leqslant \bar M/\lambda$.

Now we prove that $M\geqslant \bar M$. Consider
\[
\delta w^{\delta}_i+H_i(x,y,Dw^{\delta}_i,\textbf{0})=0\quad \textrm{for}\ y\in\mathbb T^n, \ i\in\{1,\dots,m\}.
\]
It is direct to check that $M/\delta$ (resp., $-M/\delta$) is a supersolution (resp., subsolution) of the above equation. By the comparison principle, we have 
$|\delta w^\delta_i|\leqslant \sup_{x\in\mathbb R^n,\ y\in\mathbb T^n}|\textbf{H}(x,y,0,\textbf{0})|$, which implies that $M\geqslant \bar M$. Here we recall that $\delta w^\delta_i \to -\bar H_i(x,0,\textbf{0})$ as $\delta\to 0$. 

A direct calculation shows that
\[\bar{H}_i(x,D\bar u_i(x),0)\leqslant \bar H_i(x,D\bar u_i(x),\bar{\mathbf{u}})+\Theta |\bar{\mathbf{u}}|=-\lambda \bar u_i(x)+\Theta |\bar{\mathbf{u}}|\leqslant \frac{\lambda+\Theta}{\lambda}\bar M\]
holds almost everywhere. By (H2) we get \eqref{Du(x)}.
\end{proof}

Since $\bar u_i(x)$ is a solution to
\[
\lambda u(x)+\bar H_i(x,Du(x),\bar{\mathbf{u}}(x))=0,
\]
we have the following implicit representation formula.
\begin{lemma}\label{M2}
We have
\[\bar u_i(x)=\inf_{\gamma\in\cC(x)}\int_{-\infty}^0e^{\lambda s}\bar L_i(\gamma(s),\dot\gamma(s),\bar{\mathbf{u}}(\gamma(s)))\, ds,
\]
where we denote by
$\cC(x)$
the set of all trajectories $\gamma\in \AC((-\infty,0])$
such that $\gamma(0)=x$.
Moreover, the infimum is achieved. Let $\bar {\gamma}:(-\infty,0]\to \mathbb R^n$ be a minimizer, then there is a constant $M_0$ depending only on $\bar{\mathbf{H}}$ and $\lambda$ such that $\|\dot{\bar{\gamma}}\|_{\Li}\leqslant M_0$.
\end{lemma}

In a similar way to that of \eqref{def:u1} we define the functions $u^1_i:\R^n\to\R$ by
\[
u^1_i(x):=\inf_{\gamma\in\cC(x)}\int_{-\infty}^0e^{\lambda s}L_i\Big(\gamma(s),\frac{\gamma(s)}{\varepsilon},\dot\gamma(s),\bar{\mathbf{u}}(\gamma(s))\Big)\, ds
\quad\text{for} \ i\in\{1,\ldots,m\}.
\]
It is standard to see that $u^1_i$ is the solution of
\[\lambda u(x)+H_i\Big(x,\frac{x}{\varepsilon},Du(x),\bar{\mathbf{u}}(x)\Big)=0\quad \textrm{for}\ x\in \mathbb R^n.\]
Let $\gamma_1:(-\infty,0]\to \mathbb R^n$ be a minimizer curve of $u^1_i(x)$. Then, there is a constant $M_0$ depending only on $\mathbf{H}$ and $\lambda$ such that $\|\dot{\gamma}_1\|_{\Li}\leqslant M_0$.

\begin{lemma}
For each $\varepsilon\in (0,\frac{1}{\lambda^2})$, there is a constant $C>0$ depending only on $n$, $\mathbf{H}$ and $\lambda$ such that
\[
|\mathbf{u}^1(x)-\bar{\mathbf{u}}(x)|
\leqslant C\sqrt{\varepsilon}\quad \text{for} \ x\in\R^n.
\]
\end{lemma}
\begin{proof}
Let $M$ be the constant defined in Lemma \ref{bd}, and set $\mathbf{H}^M:=\mathbf{H}-\mathbf{M}$, where we set $\mathbf{M}=(M,\ldots,M)$.
Let $\bar{\mathbf{H}}^M$ be the effective Hamiltonian for $\mathbf{H}^M$, which coincides with $\bar{\mathbf{H}}-\mathbf{M}$. By (H6), $\mathbf{u}^1_M:=\mathbf{u}^1+\frac{\mathbf{M}}{\lambda}$ and
$\bar{\mathbf{u}}_M:=\bar{\mathbf{u}}+\frac{\mathbf{M}}{\lambda}$ are, respectively, the solutions of
\[
\lambda u(x)+H^M_i\Big(x,\frac{x}{\varepsilon},Du(x),\bar{\mathbf{u}}_M(x)\Big)=0\quad \text{for}\ x\in\mathbb R^n,\ i\in\{1,\ldots,m\},
\]
and \eqref{hj0} with $\bar{\mathbf{H}}$ replaced by $\bar{\mathbf{H}}_M$. The additional property of the Hamiltonian $\mathbf{H}^M$ is that its Lagrangian and effective Lagrangian are nonnegative. In fact, a direct calculation shows that
\begin{align*}
L^M_i(x,y,v,\bar{\mathbf{u}}_M(x))&=\sup_{p\in\mathbb R^n}(p\cdot v-H_i(x,y,p,\bar{\mathbf{u}}_M(x))+M)
\\ &=\sup_{p\in\mathbb R^n}(p\cdot v-H_i(x,y,p,\bar{\mathbf{u}}(x)))+M
\\ &\geqslant -H_i(x,y,0,0)-\Theta \frac{M}{\lambda}+M=\frac{\lambda-\Theta}{\lambda}M-H_i(x,y,0,0)\geqslant 0,
\end{align*}
where $|\bar{\textbf{u}}(x)|\leq \bar M/\lambda\leq M/\lambda$ is used in the first inequality.
Since $\mathbf{u}^\varepsilon_M-\bar{\mathbf{u}}_M=\mathbf{u}^\varepsilon-\bar{\mathbf{u}}$, it suffices to prove that $|\mathbf{u}^\varepsilon_M-\bar{\mathbf{u}}_M|\leqslant C\sqrt{\varepsilon}$. Therefore, in the following, we assume that $\mathbf{L}(x,y,v,\bar{\mathbf{u}}(x))$ and $\bar{\mathbf{L}}(x,v,\bar{\mathbf{u}}(x))$ are nonnegative without loss of generality.

We only prove that
\[
u^1_i(x)-\bar u_i(x)\geqslant -C\sqrt{\varepsilon}
\quad\text{for} \ x\in\R^n, \ i\in\{1,\ldots,m\},
\]
since we can prove the reverse inequality in a similar way.
We consider a partition
\[0=t_0\geqslant -t_1\geqslant -t_2\geqslant\dots\geqslant -t_k\geqslant \dots\]
of the interval $(-\infty,0]$ with $t_k\to +\infty$ as $k\to+\infty$.

For each $k\in\N\cup\{0\}$, by Lemma \ref{lem:HJ} we have
\begin{equation}\label{1'}
\begin{aligned}
&\int^{-t_k}_{-t_{k+1}}e^{\lambda s}L_i\Big(\gamma_1(s),\frac{\gamma_1(s)}{\varepsilon},\dot\gamma_1(s),\bar{\mathbf{u}}(\gamma_1(s))\Big)\, ds
\\ \geqslant &\int^{-t_k}_{-t_{k+1}}e^{\lambda s} L_i\Big(\gamma_1(-t_k),\frac{\gamma_1(s)}{\varepsilon},\dot\gamma_1(s),\bar{\mathbf{u}}(\gamma_1(-t_k))\Big)\, ds-E_k
\\ \geqslant & e^{-\lambda t_{k+1}}m^\varepsilon(-t_{k+1},-t_k,\gamma_1(-t_{k+1}),\gamma_1(-t_k);\gamma_1(-t_k),\bar{\mathbf{u}}(\gamma_1(-t_k)))-E_k
\\ \geqslant & e^{-\lambda t_{k+1}}\bar m(-t_{k+1},-t_k,\gamma_1(-t_{k+1}),\gamma_1(-t_k);\gamma_1(-t_k),\bar{\mathbf{u}}(\gamma_1(-t_k)))-Ce^{-\lambda t_{k+1}}\varepsilon-E_k,
\end{aligned}
\end{equation}
where
\[E_k:=\int^{-t_k}_{-t_{k+1}}e^{\lambda s}\bigg[\textrm{Lip}(H)|\gamma_1(s)-\gamma_1(-t_k)|+\Theta|D\bar{\mathbf{u}}|\cdot|\gamma_1(s)-\gamma_1(-t_k)|\bigg]\, ds.\]

Now, take $\gamma^k\in\AC([-t_{k+1},-t_k])$ with $\gamma^k(-t_k)=\gamma_1(-t_k)$ and $\gamma^k(-t_{k+1})=\gamma_1(-t_{k+1})$ satisfying
\begin{align*}
&\int^{-t_k}_{-t_{k+1}} \bar L_i(\gamma_1(-t_k),\dot\gamma^k(s),\bar{\mathbf{u}}(\gamma_1(-t_k)))\, ds
\\ =&\bar m(-t_{k+1},-t_k,\gamma_1(-t_{k+1}),\gamma_1(-t_k);\gamma_1(-t_k),\bar{\mathbf{u}}(\gamma_1(-t_k))).
\end{align*}
Then,
\begin{equation}\label{2'}
\begin{aligned}
&e^{-\lambda t_{k+1}}\bar m(-t_{k+1},-t_k,\gamma_1(-t_{k+1}),\gamma_1(-t_k);\gamma_1(-t_k),\bar{\mathbf{u}}(\gamma_1(-t_k)))
\\ \geqslant &e^{-\lambda t_{k+1}}\int^{-t_k}_{-t_{k+1}} \bar L_i(\gamma^k(s),\dot\gamma^k(s),\bar{\mathbf{u}}(\gamma^k(s)))\, ds-e^{-\lambda t_{k+1}}\bar E_k
\\ \geqslant &e^{-\lambda(t_{k+1}-t_k)}\int_{-t_{k+1}}^{-t_k}e^{\lambda s}\bar L_i(\gamma^k(s),\dot\gamma^k(s),\bar{\mathbf{u}}(\gamma^k(s)))\, ds-e^{-\lambda t_{k+1}}\bar E_k,
\end{aligned}
\end{equation}
where
\[\bar E_k
:=
\int^{-t_k}_{-t_{k+1}}\bigg[\textrm{Lip}(H)|\gamma^k(s)-\gamma_1(-t_k)|+\Theta|D\bar{\mathbf{u}}|\cdot|\gamma^k(s)-\gamma_1(-t_k)|\bigg]\, ds.
\]

By Lemma \ref{M2} and \eqref{Du(x)}, we have
\[|E_k|\leqslant \int^{-t_k}_{-t_{k+1}}e^{\lambda s}\bigg[\textrm{Lip}(H)M_0|s+t_k|+\Theta CM_0|s+t_k|\bigg] ds\leqslant e^{-\lambda t_k}K(t_{k+1}-t_k)^2,\]
where
\[
K:=\frac{1}{2}\big(\textrm{Lip}(H)M_0+\Theta CM_0\big).
\]
Similarly, there is $K>0$ such that $\bar E_k\leqslant K(t_{k+1}-t_k)^2$.

Now, combining \eqref{1'} and \eqref{2'} we obtain
\begin{align*}
u^1_i(x)&=\sum_{k=0}^{+\infty} \int^{-t_k}_{-t_{k+1}}e^{\lambda s}L_i\Big(\gamma_1(s),\frac{\gamma_1(s)}{\varepsilon},\dot\gamma_1(s),\bar{\mathbf{u}}(\gamma_1(s))\Big)\, ds
\\ &\geqslant e^{-\lambda \sup_{k}(t_{k+1}-t_k)}\sum_{k=0}^{+\infty}\int_{-t_{k+1}}^{-t_k}e^{\lambda s}\bar L_i(\gamma^k(s),\dot\gamma^k(s),\bar{\mathbf{u}}(\gamma^k(s)))\, ds
\\ &\quad-\sum_{k=0}^{+\infty}(E_k+e^{-\lambda t_{k+1}}\bar E_k+Ce^{-\lambda t_{k+1}}\varepsilon)
\\ &\geqslant e^{-\lambda \sup_{k}(t_{k+1}-t_k)}\bar u_i(x)-\sum_{k=0}^{+\infty}(E_k+e^{-\lambda t_{k+1}}\bar E_k+Ce^{-\lambda t_{k+1}}\varepsilon)
\end{align*}
Set $t_k=k\sqrt{\varepsilon}$ for $k\in\N\cup\{0\}$, and then
\[
u^1_i(x)\geqslant e^{-\lambda \sqrt{\varepsilon}}\bar u_i(x)-(2K+C)\frac{\sqrt{\varepsilon}}{\lambda}\sum_{k=0}^{+\infty}\lambda \sqrt{\varepsilon}e^{-k\lambda\sqrt{\varepsilon}}.
\]
By Lemma \ref{bd}, we have $M\geqslant \bar M$, it is easy to check that $\mathbf{0}$ is a subsolution of \eqref{hj0} when $\bar{\mathbf{H}}$ is replaced by $\bar{\mathbf{H}}-M$. Then we have $\bar u_i\geqslant 0$ by Lemma \ref{cp}. By Lemma \ref{bd},
\[e^{-\lambda\sqrt{\varepsilon}}\bar u_i(x)-\bar u_i(x)\geqslant -\lambda\sqrt{\varepsilon}\bar u_i(x)\geqslant -M\sqrt{\varepsilon}.\]
We also have
\[\sum_{k=0}^{+\infty}\lambda \sqrt{\varepsilon}e^{-k\lambda\sqrt{\varepsilon}}=\frac{\lambda \sqrt{\varepsilon}}{1-e^{-\lambda \sqrt{\varepsilon}}},\]
which is bounded by a constant $\bar C>0$ independent of $\varepsilon$ when $\lambda \sqrt{\varepsilon}<1$. We then conclude that
\[
u^1_i(x)\geqslant \bar u_i(x)-\bigg(M+\frac{(2K+C)\bar C}{\lambda}\bigg)\sqrt{\varepsilon}.
\]
\end{proof}

Next, we define a family of functions $\{\mathbf{u}^\ell\}_{\ell\in\N}$ by
\[
u^{\ell+1}_i(x)=\inf_{\gamma\in\cC(x)}
\int_{-\infty}^0 e^{\lambda s} L_i\Big(\gamma(s),\frac{\gamma(s)}{\varepsilon},\dot\gamma(s),\mathbf{u}^\ell(\gamma(s))\Big)\, ds
\quad \text{for} \ x\in\R^n, \ \ell\in\N.
\]
Let $\gamma_{\ell+1}:(-\infty,0]\to\mathbb R^n$ be a curve achieving the above infimum. It is standard to see that $u^{\ell+1}_i(x)$ is continuous, and is the solution to
\[\lambda u(x)+H_i\Big(x,\frac{x}{\varepsilon},Du(x),\mathbf{u}^{\ell}(x)\Big)=0\quad \textrm{for}\ x\in \mathbb R^n.\]

\begin{lemma}\label{unx}
For each $\ell=1,2,\dots$, we have
\begin{equation}\label{nx}
 \|\mathbf{u}^{\ell+1}-\mathbf{u}^\ell\|_{L^\infty}\leqslant \bigg(\frac{\Theta}{\lambda}\bigg)^\ell C\sqrt{\varepsilon}.
\end{equation}
\end{lemma}
\begin{proof}
We prove by induction. For $\ell=1$, we have
\begin{align*}
&u^2_i(x)-u^1_i(x)
\\ \leqslant & \int_{-\infty}^0 e^{\lambda s}\bigg[L_i\Big(\gamma_1(s),\frac{\gamma_1(s)}{\varepsilon},\dot\gamma_1(s),\mathbf{u}^1(\gamma_1(s))\Big)-L_i\Big(\gamma_1(s),\frac{\gamma_1(s)}{\varepsilon},\dot\gamma_1(s),\bar{\mathbf{u}}(\gamma_1(s))\Big)\bigg]\, ds
\\ \leqslant &\int_{-\infty}^0 e^{\lambda s}\Theta|\mathbf{u}^1(\gamma_1(s))-\bar{\mathbf{u}}(\gamma_1(s))|\, ds\leqslant \Theta C\sqrt{\varepsilon}\int_{-\infty}^0e^{\lambda s}\, ds=\frac{\Theta}{\lambda} C\sqrt{\varepsilon}.
\end{align*}
By symmetry, we conclude that
\[|u^2_i(x)-u^1_i(x)|\leqslant \frac{\Theta}{\lambda} C\sqrt{\varepsilon}.\]
Assume that \eqref{nx} holds for $\ell=k$. For $\ell=k+1$, we have
\begin{align*}
&u^{k+2}_i(x)-u^{k+1}_i(x)
\\ \leqslant & \int_{-\infty}^0 e^{\lambda s}\bigg[L_i\Big(\gamma_{k+1}(s),\frac{\gamma_{k+1}(s)}{\varepsilon},\dot\gamma_{k+1}(s),\mathbf{u}^{k+1}(\gamma_{k+1}(s))\Big)
\\ &-L_i\Big(\gamma_{k+1}(s),\frac{\gamma_{k+1}(s)}{\varepsilon},\dot\gamma_{k+1}(s),\mathbf{u}^k(\gamma_{k+1}(s))\Big)\bigg]\, ds
\\ \leqslant &\int_{-\infty}^0 e^{\lambda s}\Theta|\mathbf{u}^{k+1}(\gamma_{k+1}(s))-\mathbf{u}^{k}(\gamma_{k+1}(s))|\, ds\leqslant \Theta\bigg(\frac{\Theta}{\lambda}\bigg)^k C\sqrt{\varepsilon}\int_{-\infty}^0 e^{\lambda s}\, ds
=\bigg(\frac{\Theta}{\lambda}\bigg)^{k+1} C\sqrt{\varepsilon}.
\end{align*}
By symmetry, we conclude that
\[|u^{k+2}_i(x)-u^{k+1}_i(x)|\leqslant \bigg(\frac{\Theta}{\lambda}\bigg)^{k+1} C\sqrt{\varepsilon}.\]
This completes the proof.
\end{proof}

\begin{proof}[Proof of Theorem {\rm\ref{thm2}}]
By \eqref{nx}, the sequence $\{\mathbf{u}^\ell\}_{\ell\in\N}$ is a Cauchy sequence in $C_b(\mathbb R^n;\mathbb R^m)$ when $\lambda>\Theta$. Then $\mathbf{u}^\ell\to \mathbf{u}^\infty$ uniformly on $\mathbb R^n$ as $\ell\to+\infty$. By the stability of viscosity solutions, $\mathbf{u}^\infty$ is the solution $\mathbf{u}^\varepsilon$ of \eqref{hje}. By Lemma \ref{unx}, we have
\begin{align*}
|\textbf{u}^\ell(x)-\bar{\textbf{u}}(x)|&\le\sum_{j=2}^{\ell}|\textbf{u}^j(x)-\textbf{u}^{j-1}(x)|+|\textbf{u}^1(x)-\bar{\textbf{u}}(x)|
\\ &\le \sum_{j=0}^{\ell-1}\bigg(\frac{\Theta}{\lambda}\bigg)^jC\sqrt{\varepsilon}\le \frac{\lambda}{\lambda-\Theta}C\sqrt{\varepsilon},
\end{align*}
which implies 
\[|\mathbf{u}^\varepsilon(x)-\bar{\mathbf{u}}(x)|\leqslant \frac{\lambda}{\lambda-\Theta}C\sqrt{\varepsilon}
\quad\text{for all} \ x\in\R^n.
\]
This completes the proof.
\end{proof}

\section*{Acknowledgements}

The authors thank Yuxi Han, Jiwoong Jang, Hung V. Tran for their valuable comments to the first draft of the paper.
The work of HM was partially supported by the JSPS grants: KAKENHI
\#21H04431, \#22K03382, \#24K00531.
Part of this work was done during a visit of PN in
the university of Tokyo.
He acknowledges the hospitality of the university. 
We also would like to thank the anonymous referee for many helpful suggestions. 

\section*{Declarations}

\noindent {\bf Conflict of interest statement:} The authors state that there is no conflict of interest.

\medskip

\noindent {\bf Data availability statement:} Data sharing not applicable to this article as no datasets were generated or analysed during the current study.


\end{document}